\documentclass[12pt, reqno]{amsart}

\usepackage{comment}
\input xy
\xyoption{all}

\begin{document}

\newtheorem{theorem}{Theorem}[section]
\newtheorem{prop}[theorem]{Proposition}
\newtheorem{lemma}[theorem]{Lemma}
\newtheorem{corollary}[theorem]{Corollary}
\newtheorem{cor}[theorem]{Corollary}
\newtheorem{defn}[theorem]{Definition}
\newtheorem{conj}[theorem]{Conjecture}
\newtheorem{rmk}[theorem]{Remark}
\newtheorem{qn}[theorem]{Question}
\newtheorem{claim}[theorem]{Claim}
\newtheorem{defth}[theorem]{Definition-Theorem}

\newcommand{\boundary}{\partial}
\newcommand{\C}{{\mathbb C}}
\newcommand\U{{\mathbb U}}
 \newcommand\Hyp{{\mathbb H}}
\newcommand\D{{\mathbb D}}
\newcommand\Z{{\mathbb{Z}}}
\newcommand\R{{\mathbb R}}
\newcommand\Q{{\mathbb Q}}
\newcommand\E{{\mathbb E}}
\newcommand{\ints}{{\mathbb{Z}}}
\newcommand{\natls}{{\mathbb N}}
\newcommand{\ratls}{{\mathbb Q}}
\newcommand{\rls}{{\mathbb R}}
\newcommand{\proj}{{\mathbb P}}
\newcommand{\lhp}{{\mathbb L}}
\newcommand{\tube}{{\mathbb T}}
\newcommand{\cusp}{{\mathbb P}}
\newcommand\AAA{{\mathcal A}}
\newcommand\BB{{\mathcal B}}
\newcommand\CC{{\mathcal C}}
\newcommand\DD{{\mathcal D}}
\newcommand\EE{{\mathcal E}}
\newcommand\FF{{\mathcal F}}
\newcommand\GG{{\mathcal G}}
\newcommand\HH{{\mathcal H}}
\newcommand\II{{\mathcal I}}
\newcommand\JJ{{\mathcal J}}
\newcommand\KK{{\mathcal K}}
\newcommand\LL{{\mathcal L}}
\newcommand\MM{{\mathcal M}}
\newcommand\NN{{\mathcal N}}
\newcommand\OO{{\mathcal O}}
\newcommand\PP{{\mathcal P}}
\newcommand\QQ{{\mathcal Q}}
\newcommand\RR{{\mathcal R}}
\newcommand\SSS{{\mathcal S}}
\newcommand\TT{{\mathcal T}}
\newcommand\UU{{\mathcal U}}
\newcommand\VV{{\mathcal V}}
\newcommand\WW{{\mathcal W}}
\newcommand\XX{{\mathcal X}}
\newcommand\YY{{\mathcal Y}}
\newcommand\ZZ{{\mathcal{Z}}}
\newcommand\CH{{\CC\HH}}
\newcommand\PEY{{\PP\EE\YY}}
\newcommand\MF{{\MM\FF}}
\newcommand\RCT{{{\mathcal R}_{CT}}}
\newcommand\RCTT{{{\mathcal R}^2_{CT}}}
\newcommand\PMF{{\PP\kern-2pt\MM\FF}}
\newcommand\FL{{\FF\LL}}
\newcommand\PML{{\PP\kern-2pt\MM\LL}}
\newcommand\GL{{\GG\LL}}
\newcommand\Pol{{\mathcal P}}
\newcommand\half{{\textstyle{\frac12}}}
\newcommand\Half{{\frac12}}
\newcommand\Mod{\operatorname{Mod}}
\newcommand\Area{\operatorname{Area}}
\newcommand\ep{\epsilon}
\newcommand\hhat{\widehat}
\newcommand\Proj{{\mathbf P}}
\newcommand\til{\widetilde}
\newcommand\length{\operatorname{length}}
\newcommand\tr{\operatorname{tr}}
\newcommand\gesim{\succ}
\newcommand\lesim{\prec}
\newcommand\simle{\lesim}
\newcommand\simge{\gesim}
\newcommand{\simmult}{\asymp}
\newcommand{\simadd}{\mathrel{\overset{\text{\tiny $+$}}{\sim}}}
\newcommand{\ssm}{\setminus}
\newcommand{\diam}{\operatorname{diam}}
\newcommand{\pair}[1]{\langle #1\rangle}
\newcommand{\T}{{\mathbf T}}
\newcommand{\inj}{\operatorname{inj}}
\newcommand{\pleat}{\operatorname{\mathbf{pleat}}}
\newcommand{\short}{\operatorname{\mathbf{short}}}
\newcommand{\vertices}{\operatorname{vert}}
\newcommand{\collar}{\operatorname{\mathbf{collar}}}
\newcommand{\bcollar}{\operatorname{\overline{\mathbf{collar}}}}
\newcommand{\I}{{\mathbf I}}
\newcommand{\tprec}{\prec_t}
\newcommand{\fprec}{\prec_f}
\newcommand{\bprec}{\prec_b}
\newcommand{\pprec}{\prec_p}
\newcommand{\ppreceq}{\preceq_p}
\newcommand{\sprec}{\prec_s}
\newcommand{\cpreceq}{\preceq_c}
\newcommand{\cprec}{\prec_c}
\newcommand{\topprec}{\prec_{\rm top}}
\newcommand{\Topprec}{\prec_{\rm TOP}}
\newcommand{\fsub}{\mathrel{\scriptstyle\searrow}}
\newcommand{\bsub}{\mathrel{\scriptstyle\swarrow}}
\newcommand{\fsubd}{\mathrel{{\scriptstyle\searrow}\kern-1ex^d\kern0.5ex}}
\newcommand{\bsubd}{\mathrel{{\scriptstyle\swarrow}\kern-1.6ex^d\kern0.8ex}}
\newcommand{\fsubeq}{\mathrel{\raise-.7ex\hbox{$\overset{\searrow}{=}$}}}
\newcommand{\bsubeq}{\mathrel{\raise-.7ex\hbox{$\overset{\swarrow}{=}$}}}
\newcommand{\tw}{\operatorname{tw}}
\newcommand{\base}{\operatorname{base}}
\newcommand{\trans}{\operatorname{trans}}
\newcommand{\rest}{|_}
\newcommand{\bbar}{\overline}
\newcommand{\UML}{\operatorname{\UU\MM\LL}}
\newcommand{\EL}{\mathcal{EL}}
\newcommand{\tsum}{\sideset{}{'}\sum}
\newcommand{\tsh}[1]{\left\{\kern-.9ex\left\{#1\right\}\kern-.9ex\right\}}
\newcommand{\Tsh}[2]{\tsh{#2}_{#1}}
\newcommand{\qeq}{\mathrel{\approx}}
\newcommand{\Qeq}[1]{\mathrel{\approx_{#1}}}
\newcommand{\qle}{\lesssim}
\newcommand{\Qle}[1]{\mathrel{\lesssim_{#1}}}
\newcommand{\simp}{\operatorname{simp}}
\newcommand{\vsucc}{\operatorname{succ}}
\newcommand{\vpred}{\operatorname{pred}}
\newcommand\fhalf[1]{\overrightarrow {#1}}
\newcommand\bhalf[1]{\overleftarrow {#1}}
\newcommand\sleft{_{\text{left}}}
\newcommand\sright{_{\text{right}}}
\newcommand\sbtop{_{\text{top}}}
\newcommand\sbot{_{\text{bot}}}
\newcommand\sll{_{\mathbf l}}
\newcommand\srr{_{\mathbf r}}
\newcommand\geod{\operatorname{\mathbf g}}
\newcommand\mtorus[1]{\boundary U(#1)}
\newcommand\A{\mathbf A}
\newcommand\Aleft[1]{\A\sleft(#1)}
\newcommand\Aright[1]{\A\sright(#1)}
\newcommand\Atop[1]{\A\sbtop(#1)}
\newcommand\Abot[1]{\A\sbot(#1)}
\newcommand\boundvert{{\boundary_{||}}}
\newcommand\storus[1]{U(#1)}
\newcommand\Momega{\omega_M}
\newcommand\nomega{\omega_\nu}
\newcommand\twist{\operatorname{tw}}
\newcommand\modl{M_\nu}
\newcommand\MT{{\mathbb T}}
\newcommand\Teich{{\mathcal T}}

\newcommand{\etalchar}[1]{$^{#1}$}

\renewcommand{\Re}{\operatorname{Re}}
\renewcommand{\Im}{\operatorname{Im}}

\title[Three Manifold Groups, K\"ahler Groups and Complex Surfaces]{Three
Manifold Groups, K\"ahler Groups\\ and Complex Surfaces}

\author[I. Biswas]{Indranil Biswas}

\address{School of Mathematics, Tata Institute of Fundamental
Research, Homi Bhabha Road, Bombay 400005, India}

\email{indranil@math.tifr.res.in}

\author[M. Mj]{Mahan Mj}

\address{RKM Vivekananda University, Belur Math, WB-711 202, India}

\email{mahan.mj@gmail.com; mahan@rkmvu.ac.in}

\author[H. Seshadri]{Harish Seshadri}

\address{Indian Institute of Science, Bangalore 560003, India}

\email{harish@math.iisc.ernet.in}


\date{\today}

\thanks{Research of the
second author is partly supported by  CEFIPRA Indo-French Research Grant 4301-1}

\begin{abstract} Let  $G$ be a  K\"ahler group admitting a short  exact sequence  
$$ 1 \longrightarrow N  \longrightarrow G
\longrightarrow Q \longrightarrow 1$$
where $N$ is finitely generated. \\
i) Then $Q$ cannot be non-nilpotent solvable. \\
ii) Suppose in addition that $Q$ satisfies one of the following: \\
a)  $Q$ admits a discrete faithful non-elementary action on ${\mathbb{H}}^n$ for some $n \geq 2$ \\
b)  $Q$ admits a discrete faithful non-elementary minimal action on  a simplicial tree with more than two ends.\\
c) $Q$ admits a (strong-stable) cut $R$ such that the intersection of all conjugates of $R$ is trivial \\ 
Then $G$ is virtually a surface group. 

It follows that if $Q$  is
infinite, not virtually cyclic, and  is the
fundamental group of some closed 3-manifold, then  $Q$ contains as a finite
index subgroup either a finite index subgroup of
the $3$-dimensional Heisenberg group or the fundamental group of the
Cartesian product of a closed
oriented surface of positive genus and the circle.
As a corollary, we obtain a new proof of a theorem of Dimca and Suciu
in \cite{ds} by taking $N$ to be the trivial group.

If instead, $G$ is the fundamental group of a compact complex surface, and $N$ is finitely presented, then we show
that $Q$ must contain the fundamental group of a Seifert-fibered three
manifold as a finite index subgroup, and $G$ contains as a finite
index subgroup the fundamental group of an elliptic fibration.

We also give an example showing that the relation of quasi-isometry
does not preserve  K\"ahler groups. This gives a negative answer to a
question of Gromov which asks whether K\"ahler groups can be
characterized by their asymptotic geometry.
\end{abstract}

\maketitle

\tableofcontents

\section{Introduction}

In this paper,
we shall be concerned mainly with the following general set-up:
$$
1 \longrightarrow N \stackrel{i}{\longrightarrow} G \stackrel{q}{\longrightarrow} Q \longrightarrow 1
\ \ \ \ \ \ \ \ \ \ \ \ \ \ \ \ \ \ \ \ \ \ \ (\ast )
$$
is an exact sequence of finitely generated groups. We shall further assume
that $Q$ is infinite and not virtually cyclic. The group
$G$ will either be a K\"ahler group, i.e., the fundamental group of a
compact K\"ahler manifold,
or the fundamental group of a compact complex surface. We fix the letters $N, G,  i, q$ to have this connotation
throughout this paper. We investigate the restrictions that these
assumptions impose on the nature of $G$ and $Q$. It will turn out that if in addition, $Q$ is assumed to be 
the fundamental group of a 
3-manifold, then 
the existence of such an exact sequence $(\ast )$ with $N$ finitely presented shall force $Q$ to be the fundamental group of a Seifert-fibered
3-manifold. 

The following   technical Proposition is the starting point of the first part of the paper dealing with  K\"ahler groups. Various 
restrictions are imposed on $Q$ and we deduce that $G$ is virtually a surface group. Item  (i)  follows from 
  \cite{delz} \cite{brudnyi-solvq}, Item ii(a) from \cite{delzant-py}, \cite{carlson-toledo-pihes} and
Items  ii(b), (c) from the cut-K\"ahler Theorem of Delzant-Gromov \cite{dg}.

\begin{prop}\label{main1t} Let  $G$ be a  K\"ahler group admitting a short  exact sequence  
$$ 1 \longrightarrow N  \longrightarrow G
\longrightarrow Q \longrightarrow 1$$
where $N$ is finitely generated. \\
i) Then $Q$ cannot be non-nilpotent solvable. \\
ii) Suppose in addition that $Q$ satisfies one of the following: \\
a)  $Q$ admits a discrete faithful non-elementary action on ${\mathbb{H}}^n$ for some $n \geq 2$ \\
b)  $Q$ admits a discrete faithful non-elementary minimal action on a simplicial tree with more than two ends.\\
c) $Q$ admits a (strong-stable) cut $R$ such that the intersection of all conjugates of $R$ is trivial \\ 
Then $G$ is virtually a surface group. 
 \end{prop}

As a consequence we have the following
(see Theorem \ref{main1}):

\begin{theorem}\label{main1i}
Let $$ 1 \longrightarrow N \stackrel{i}{\longrightarrow} G
\stackrel{q}{\longrightarrow} Q \longrightarrow 1$$  be the exact
sequence $(\ast )$ such that $G$ is a K\"ahler group and $Q$  is an
infinite, not virtually cyclic, 
fundamental group of some closed 3-manifold.
Then there exists a finite index subgroup $Q^\prime$ of $Q$ such that
either
$Q^\prime$ is a finite index subgroup of the $3$-dimensional Heisenberg
group or $Q^\prime = \pi_1(\Sigma \times S^1)$, where $\Sigma$ is a
closed oriented surface of positive genus.
 \end{theorem}

(The $3$-dimensional Heisenberg group consists of the unipotent upper
triangular elements of $\text{GL}(3,{\mathbb Z})$.)

Let $Q$ be
the fundamental group of a closed 3-manifold which is infinite.
Donaldson and Goldman conjectured that $Q$ cannot be a K\"ahler group. This conjecture was  proved by
Dimca and Suciu \cite{ds} and by Delzant (see Theorem \ref{thdel}); later a different proof was given by
Kotschick \cite{kotschick}. We obtain a proof of this theorem of
\cite{ds} by setting $N$ in Theorem \ref{main1i} to be the trivial group
(see Theorem \ref{dsk}) and a simple argument to exclude the $3$-dimensional Heisenberg group case.

The next theorem  deals with the case that $G$ is the fundamental group
of a compact complex surface (see Theorem \ref{main2}).

\begin{theorem}\label{main2i}
Let $$ 1 \longrightarrow N \stackrel{i}{\longrightarrow} G
\stackrel{q}{\longrightarrow} Q \longrightarrow 1$$  be the exact
sequence $(\ast
)$ such that  $G$ is the fundamental group of a compact complex surface and $Q$ is an
infinite, not virtually cyclic, 
fundamental group of some closed 3-manifold.
Then there exists a finite index subgroup $Q^\prime$ of $Q$ such that
$Q^\prime$ is the fundamental group of a Seifert-fibered 3-manifold with hyperbolic or flat base
orbifold. Also there exists
a finite index subgroup $G^\prime$ of $G$ such that $G^\prime$ is the fundamental group of an elliptic complex surface $X$
which is a circle bundle over a Seifert-fibered 3-manifold.\end{theorem}

In Theorem \ref{main2i}, setting $N$ to be the trivial group we conclude
that $Q$ is not the fundamental group of a compact complex surface if $Q$ is infinite and
not virtually cyclic (see Theorem \ref{dsk2}).

Stronger results when $X$ is of Class VII, or admits an elliptic
fibration, are given in Theorem \ref{main7} and Proposition
\ref{ellprop} respectively.

As a consequence of Theorems \ref{main1i} and \ref{main2i} we also get the following result (see Theorem
\ref{3times}), the first part of which was proven by J. Hillman ~\cite{hill}  based on work of Wall ~\cite{wall}. The second part
follows from Theorem \ref{main1i} and the fact that the product of the Heisenberg group with $\mathbb Z$ has 
$\mathbb{Z}  \oplus  \mathbb{Z} \oplus \mathbb{Z}$ as its abelianization.

\begin{theorem}\label{3timesi}
Let $M$ be a closed orientable 3 manifold. Then

(i) $M \times S^1$ admits a
complex structure if and only if $M$ is Seifert fibered.

(ii) $M \times S^1$ admits a
K\"ahler complex structure if and only if $M=\Sigma \times S^1$ where $\Sigma$ is a   compact surface.
\end{theorem}

This can be compared with Thurston's Conjecture, proved by S. Friedl and S. Vidussi \cite{fv},
that $M \times S^1$ admits a symplectic structure if and only if $M$ fibers over a circle. Since a K\"ahler
structure is both symplectic and complex, (ii) above is natural given (i) and the Friedl-Vidussi result. \vspace{3mm}

This paper can be regarded as giving a unified approach  generalizing the results of Dimca-Suciu and that of Hillman,
by taking $N$ to be trivial in Theorem \ref{main1i} and $N= {\mathbb Z}$ in Theorem \ref{main2i}. We would like to point out that the techniques we use are  quite different
from Dimca-Suciu's as well as Hillman's, which in turn are very different from each other.

 Finally, we give an example demonstrating the fact that K\"ahlerness is not preserved by quasi-isometries. This
gives a strong negative answer to the following
question of Gromov (Problem on page 209  of \cite{gromov-ai}):
\vspace{2mm}

{\it Can one characterize K\"ahler groups by their asymptotic geometry?}
\vspace{2mm}

 The above question makes sense only within the holomorphic category
 (as the fundamental group of a 2-torus is an index two
subgroup of the fundamental group of a Klein bottle). However, if a group $G$ acts freely by holomorphic automorphisms on a simply connected
(non-compact) K\"ahler manifold $X$ and there exists a finite index subgroup $G_1$ of $G$ such that
the quotient space $X/G_1$ is K\"ahler, then $G$ is a  K\"ahler group (cf. \cite{arapura-hom}).

In the light of this, one makes the above question precise by asking
if quasi-isometric to a K\"ahler group implies commensurable to a
(possibly different) K\"ahler group. This turns out to be false. We give examples of
aspherical complex surfaces with quasi-isometric fundamental groups,
some of which are K\"ahler and others not.

The main tools used in this paper are: \\
(1) A detailed structure theory for  3-manifolds arising from the Geometrization Theorem of Thurston and Perelman. \\
(2) The cut-K\"ahler Theorem of Delzant and Gromov \cite{dg}.

In particular, our techniques differ from those of \cite{ds} and \cite{kotschick}.
The main difference is that we approach the problem from a
3-manifolds and Geometric Group Theory perspective
rather than a Complex Geometry perspective.

We prove the following key proposition (Case ii(c) of Proposition \ref{main1t} above)
which might be of independent interest (See
Corollary \ref{pppcor} for the proof). We first make the following definition:

Let $Q$ be a group and $R$ a subgroup with the following property: If $M$ is a compact Riemannian manifold with fundamental group $Q$ and  $ M'$  is the cover of  $M$ corresponding to $R$ then

(i) $M'$ has at least $3$ ends (i.e., the number of {\it relative ends} $e(Q,R) \ge 3$) and

(ii) $M'$ is a non-amenable metric space.

We then say that $R$ is a {\it cut subgroup} of $Q$.

\begin{prop}\label{ppp}
Let $$1 \rightarrow N \rightarrow G \rightarrow Q \rightarrow 1$$ be an exact sequence of finitely generated  groups. Assume that  $G$ is K\"ahler and
$Q$ has a finitely generated cut subgroup  $R$. Further assume that no nontrivial normal subgroup of $Q$ is contained in $R$; equivalently assume that
the intersection of all conjugates of $R$ is trivial.
Then Q is a surface group.
\end{prop}

\section{K\"ahler Groups}
 All manifolds
will be connected.

\subsection{Restrictions for K\"ahler Groups}
In this subsection, we collect together known restrictions on K\"ahler
groups that are used here. See \cite{abckt} for further details.

The following theorem imposes strong restrictions on homomorphisms from a K\"ahler group to a real hyperbolic lattice.

\begin{theorem}[\cite{delzant-py}, \cite{carlson-toledo-pihes},
\cite{sampson2},  Section 6.4 of \cite{abckt}]
Let $\Gamma$ be a K\"ahler group and $\Gamma_1$ be a non-elementary subgroup in ${\mathbb{H}}^n$, $n \geq 2$.
 Let $h: \Gamma \longrightarrow \Gamma_1$ be any  homomorphism with
infinite image. Then  $h$
factors through a fibration onto a hyperbolic 2-orbifold.\label{dp} \end{theorem}

The following Corollary establishes Case ii(a) of Proposition \ref{main1t}.

\begin{cor} \label{dpcor}
Let $$1 \longrightarrow N \stackrel{i}{\longrightarrow} G \stackrel{q}{\longrightarrow} Q \longrightarrow 1$$
 be an exact sequence of finitely generated  groups. Assume that  $G$ is K\"ahler 
and $Q$ admits a non-elementary action on ${\mathbb{H}}^n$, $n \geq 2$. Then $G = Q$ is a virtual surface group.
\end{cor}

\begin{proof} By Theorem \ref{dp}, $q$ factors through a fibration onto a hyperbolic 2-orbifold $S$, i.e. there exists $h: G \rightarrow \pi_1(S)$
and $h_1: \pi_1(S) \rightarrow Q$ such that $h_1\circ h = q$. In particular $Q= \frac{\pi_1(S)}{h\circ i(N)}$ is the quotient
of $\pi_1(S)$ by the finitely generated normal subgroup $h\circ i(N)$. Since $Q$ is infinite, it follows that  
the finitely generated normal subgroup $h\circ i(N)$ must be trivial (since any finitely generated normal subgroup
of the fundamental group of  a hyperbolic orbifold is either trivial or of finite index). The Corollary follows.
\end{proof}

The next Theorem imposes strong restrictions on homomorphisms from a K\"ahler group to solvable groups
and establishes Item (i) of Proposition \ref{main1t}.

\begin{theorem} \cite{delz} \cite{brudnyi-solvq} Let $G$ be a K\"ahler group such
that there is an exact sequence
$${1} \longrightarrow F \longrightarrow G \longrightarrow H
\longrightarrow {1}\, ,$$
where $F$ does not admit a surjective homomorphism onto $F\infty$ (the
free group on infinitely many generators), and $H$ is  solvable. Then there exist normal subgroups $H_2 \subset H_1$ of $H$ such that \\
 (a) $H_1$ has finite index in $H$,\\
(b) $H_1/H_2$ is nilpotent, and\\
(c) $H_2$ is torsion. \label{brudnyi} \end{theorem}

 The following version of the {\it Cut-K\"ahler Theorem} of Delzant and Gromov is one of the principal tools of this paper. For further details see
Sections 6, 7 and  3.8 in \cite{dg}.

\begin{theorem} \cite{dg} Let $G$ be the fundamental group $G$ of a
compact K\"ahler manifold $X$, and let $H \subset G$ be a subgroup.
Suppose that
\begin{enumerate}
\item the cover $X_1$ of $X$ corresponding to the subgroup $H$ is
non-amenable, i.e., domains in $X_1$ satisfy a linear isoperimetric
inequality, and

\item $\# (Ends(X_1))  > 2$.
\end{enumerate}
Then there is a finite cover $X^\prime$ of $X$
admitting a surjective holomorphic map $X^\prime \longrightarrow S$
to a Riemann surface $S$ with
connected fibers, such that the pullback
to
$\pi_1 (X^\prime ) \subset \pi_1(X)$ of some subgroup in $\pi_1(S)$
coincides with $H \bigcap \pi_1 (X^\prime )$. In particular, the kernel
of the
induced homomorphism $\pi_1 (X^\prime )\longrightarrow \pi_1(S)$ is
contained in $H \bigcap \pi_1 (X^\prime )$. \label{cut} \end{theorem}

\begin{rmk}\label{sur}
 \rm {The map from $X$ to $S$ induces a surjection from $\pi_1(X)$ to $\pi_1(S)$. This follows from the well-known fact that a surjective holomorphic
map between compact complex manifolds with connected fibers induces a surjection of fundamental groups.}
\end{rmk}

 As a corollary of the Cut-K\"ahler Theorem we have the following which is item ii(c) of Theorem \ref{main1t}.

\begin{corollary}\label{pppcor}
Let $$1 \rightarrow N \rightarrow G \rightarrow Q \rightarrow 1$$ be an exact sequence of finitely generated  groups. Assume that  $G$ is K\"ahler and $Q$ has a finitely generated cut subgroup $R$. Further assume that no nontrivial normal subgroup of $Q$ is contained in $R$; equivalently suppose
that the intersection of all conjugates of $R$ is trivial. Then $Q$ is virtually a surface group.
\end{corollary}

Before we prove this, we note the following simple fact:
\begin{lemma} \label{bn}
Let $X, \ Y$ be  compact Riemannian manifolds and $q: \pi_1(X) \longrightarrow
\pi_1(Y)$ a surjective homomorphism. If $Y_1$ is a  cover of $Y$ and $X_1$ the corresponding
cover of $X$ (with $\pi_1 (X_1)= q^{-1}(\pi_1(Y_1))$), then $X_1$ is
quasi-isometric to $Y_1$.
\end{lemma}

\begin{proof} This follows from the fact that the cover of $X$ corresponding to the subgroup $ker(q_\ast ) \subset \pi_1(X)$ is
quasi-isometric to the universal cover of $Y$.
\end{proof}

\begin{proof}(of Corollary \ref{pppcor}):

Let $X$ be a compact K\"ahler manifold with fundamental group $G$.
 Let $$H \,:=\, q^{-1} (R) \,\subset \,G\, ,$$ and let $X_1$
be the covering of $X$
corresponding to the subgroup $H$. Let $M$ be a compact 2-complex with fundamental group $Q$
and let  $M' $ be the cover of $M$ corresponding to $R$. Note that $M' $ has at least $3$ ends and is  non-amenable by hypothesis.
Since $M' $ and $X_1$ are quasi-isometric by Lemma
\ref{bn}, the K\"ahler manifold $X_1$ has these properties as well.
We can then apply the cut-K\"ahler Theorem \ref{cut} of Delzant-Gromov to conclude that there
 exists a finite cover $\psi$ from (a finite cover of)
$X$ to a Riemann surface $S$ with connected fibers
such that the pullback to
$\pi_1 (X) $ of some subgroup in $\pi_1(S)$ equals $H$.
In particular, the kernel  $$ker(\psi_\ast ) \subset H\, .$$

By hypothesis the normal subgroup $q(ker(\psi_\ast )) \subset R \subset Q$ must be  trivial.  Hence
$ker(\psi_\ast) \subset N $.
We therefore have an exact sequence
$$1 \longrightarrow  \frac{N}{ker(\psi_\ast)} \longrightarrow \pi_1(S)
\longrightarrow Q\longrightarrow  1$$
where $\frac{N}{ker(\psi_\ast)}$ is a finitely generated normal subgroup of $\pi_1(S)$. Hence $\frac{N}{ker(\psi_\ast)}$ is trivial and $Q= \pi_1(S)$.
\end{proof}

\smallskip

\noindent {\bf Examples:} \\
1) Recall the following definition from ~\cite{gmrs}: Let $R$ be an infinite subgroup of a group $Q$. We say that the {\it height} of $R$ in $Q$ is finite  if there is a positive integer $n$ such that the intersection of any collection of $n$ distinct conjugates of $R$ corresponding to $n$ distinct coset representatives is trivial. It follows from Corollary \ref{pppcor} that if
 $$1 \rightarrow N \rightarrow G \rightarrow Q \rightarrow 1$$ be an exact sequence of finitely generated  groups
such that  $G$ is K\"ahler and $Q$ has a finitely generated cut subgroup  $R$ of finite height, then $Q$ is virtually a surface group.

\smallskip

\noindent 2) 
Let $$1 \rightarrow N \rightarrow G \rightarrow Q \rightarrow 1$$ be an exact sequence of finitely generated  groups
such that  $G$ is K\"ahler and $Q$ admits a discrete faithful action on a simplicial tree $T$ with more than two ends such that some vertex stabilizer 
$R$ is
finitely generated. The intersection of  conjugates of $R$ is trivial as the action is faithful. Also since $T$ has more than two ends,
$R$ is a cut subgroup. It follows from Corollary \ref{pppcor} that $Q$ is virtually a surface group. This
establishes Case ii(b) of Proposition \ref{main1t}.

\smallskip

We conclude with a Proposition which might be of independent interest.

\begin{prop} Let $$1 \rightarrow N \rightarrow G \rightarrow Q \rightarrow 1$$ be an exact sequence of finitely presented  groups. Assume that
$G$ is K\"ahler and
$ Q$ is (Gromov)-hyperbolic.
If $Q$ contains a finitely presented quasiconvex codimension one subgroup $H$ (i.e. the number of relative ends $e(Q, H) \geq 2$)
then $Q$ must virtually be a surface group. \label{cor4} \end{prop}

\begin{proof}  An argument in \cite{gromov-hypgps} (see also \cite{dg}) shows that there exists $h \in Q\setminus H$ such that $H$ and $h$ generate the
free product $H*<h> =H_1$ which is quasiconvex and $e(Q,H_1) \geq 3$.
Therefore $H_1$ is a  cut subgroup. Since quasiconvex subgroups have finite height \cite{gmrs},
  $H_1$ satisfies the hypotheses of Example 1 above. \end{proof}

Similar results for relatively hyperbolic groups with codimension one relatively quasiconvex subgroups hold in the light of
a result of Hruska-Wise \cite{hr-w}.

\noindent {\bf Example:} A large class of hyperbolic small cancellation groups \cite{wise} have codimension one quasiconvex subgroups. It follows that
if such a group appears as the quotient group $Q$ in the exact sequence above, it must be virtually a surface group.

\subsection{3 Manifolds} We very briefly recall the necessary 3-manifold topology and geometry
we need and refer to \cite{hempel-book}, \cite{bon-note}, \cite{scott-geoms} for details.
A closed orientable 3-manifold is said to be prime if it cannot
expressed as a non-trivial connected sum (denoted by $\#$) of
3-manifolds.
 Let $M$ be a closed orientable 3-manifold. Then, there is a unique collection (up to permutation)
of prime  3-manifolds $M_1 \cdots , M_k$ such that $M$ is homeomorphic to $M_1 \# \cdots \# M_k$.
Further $\pi_1(M) = \pi_1(M_1) \ast \cdots \ast \pi_1(M_k)$. Cutting $M$ along essential spheres and capping off the resulting
spheres with $3$-balls, we get a unique
{\it prime decomposition} of $M$.

Now suppose that $M$ is prime. Then, up to isotopy,
there is a unique compact two dimensional submanifold $T$, each of whose components is a 2-sided essential torus, such that
 every component of $M \setminus T$ either admits a Seifert fibration
or else any  essential embedded torus in it is parallel to the boundary. Decomposing a 3-manifold along such tori will be called
the {\it torus decomposition}.

A consequence of the Geometrization Theorem (due to Thurston and
Perelman), \cite{Per1}, \cite{Per2}, \cite{Per3},
is that after prime decomposition and torus decomposition, each resulting piece admits a complete geometric structure modeled on
$\E^3$ , $S^3$ , $\Hyp^3$ , $\Hyp^2 \times \E^1$ , $S^2 \times \E^1$, $\E^2 \widetilde{\times} \E^1$, $\Hyp^2 \widetilde{\times} \E^1$ or $Sol$.
Of these, only manifolds modeled on  $S^2 \times \E^1$ contains an essential 2-sphere (see \cite{scott-geoms}).

\begin{theorem} (\cite{bon-note}, Lemma 2.1 of \cite{kl-npc})\label{new}
Let $M$ be a closed, orientable prime 3-manifold with nontrivial torus decomposition. Then there is a finite cover $M_1$ of $M$ such that the
any piece in the torus decomposition of $M_1$ is either a complete finite volume hyperbolic manifold or of the form $\Sigma \times S^1$, where $S$ is a compact surface
with at least two boundary components and of positive genus.
\end{theorem}

The following theorem is a consequence of works of Scott, Casson-Jungreis and Gabai as well as
Perelman's proof of the Poincar\'e Conjecture.

\begin{theorem}[\cite{casson-j}, \cite{gabai-cgnce}, \cite{scott-sf}]
Let $M$ be a compact, orientable, irreducible 3-manifold such
that $\pi_1(M)$ contains an infinite
cyclic normal subgroup. Then $M$ is a Seifert-fibered space. \label{sfs} \end{theorem}

The next theorem is a consequence of work of Luecke (Theorem 1 of \cite{luecke-vbetti}) and the
residual finiteness of (geometrizable) 3-manifold groups \cite{hempel-rf}.

\begin{theorem}   Let $M$ be a compact, orientable, irreducible
3-manifold such that at least
one of the following conditions is satisfied:
\begin{enumerate}
\item[(a)] $M$ contains  an incompressible     torus and   is not  a
$Sol$ manifold.
\item[(b)] $M$ is a Seifert-fibered space such that the base of the
Seifert fibration is not an elliptic orbifold.

\item[(c)] $M$ is not prime.
\end{enumerate}
Then $M$ has  a finite cover  $\rho_1: M_1 \longrightarrow M$, such that
{\rm rank}$(H_1 (M_1)) \geq 3$. Further, if $M$ is not covered by the 3-torus,
then
 for any positive integer   $k$,
there is a finite cover, $\rho_k: M_k \longrightarrow M$, such that {\rm
rank}$(H_1 (M_k)) > k$.
\label{luecke} \end{theorem}

Case (a),  where $M$ has a non-trivial torus decomposition, is due to
Luecke. Case (b) follows from the fact that fundamental
groups of Seifert manifolds admit surjections to the fundamental groups of the associated base orbifold \cite{scott-geoms}.
Case (c) is a consequence of Grushko's theorem when $M$ is not prime
(cf. Lemma 7 of \cite{kotschick}).

\medskip

\noindent {\bf Torus Decompositions and Finite Height} We shall now focus on 3-manifolds admitting a non-trivial torus decomposition and
 show that if $(\pi_1(M),\pi_1(M_1))$
is  a
 pair such that $M$ is a prime 3-manifold admitting a non-trivial torus decomposition and $M_1$ is one of the pieces, then 
{\bf $\pi_1(M_1)$
has finite height in $\pi_1(M)$}. In particular,
the pair $(\pi_1(M),\pi_1(M_1))$ satisfies the hypotheses on the pair $(Q,R)$ occurring in Corollary \ref{pppcor}. 

A special case where finite height subgroups arise is when
 groups act {\it acylindrically} on simplicial trees. Recall ~\cite{zs} that a group $Q$ is said to have a 
{\it $k$-acylindrical} action on a tree $T$ if the action is minimal and for all $g \in G \setminus \{1\}$, the fixed set of $g$ has diameter bounded by $k$. In this case, it is clear that the height of any vertex stabilizer $R$ is bounded above by $k$.

We shall, in fact,  prove the stronger
assertion that the action of
$\pi_1(M)$ on the Bass-Serre tree associated to the splitting over $\pi_1(M_1)$ is 3-acylindrical. 
It follows that such an action is faithful and minimal. 
The proof we give below is due to the referee who gave a substantial
simplification of our earlier proof. The proof uses some standard but deep facts from 3-manifold topology, namely the torus and annulus
theorems. In the next subsection we shall give a proof that uses coarse geometry only.

\begin{prop}  Let $M$ be a prime 3-manifold admitting a non-trivial
torus decomposition and let $T$ be the Bass-Serre
tree dual to the  torus  decomposition of $M$.
Then the action of $\pi_1(M)$ on $T$ is 3-acylindrical. \label{remark-torus}
\end{prop}

\begin{proof}
Suppose $g\in \pi_1(M) (=G)$ fixes a pair of  distinct edges in $T$
incident on a common vertex
$v \in T$. Then the stabilizer $G_v$ of $v$ is (isomorphic to) the
fundamental group of a piece $M_2$ of  the  torus  decomposition of $M$. We claim that
$M_2$ must be Seifert-fibered. To see this we first note that $g$
gives rise to an essential
annulus in $M_2$, i.e. a map $f: (S^1 \times I, \ S^1 \times \partial
I) \rightarrow (M_2, \partial M_2)$
such that both $\pi_1(S^1 \times I) \rightarrow \pi_1(M_2)$
and  $\pi_1(S^1 \times I, \ S^1 \times \partial I) \rightarrow
\pi_1(M_2, \partial M_2)$ are
injections. Next we recall the
Annulus Theorem (~\cite{js}, p. 132, IV 3.2): {\it  Let $N$ be a
compact irreducible 3-manifold with
$\partial N \neq \phi$. If $N$ contains an essential (immersed) annulus then
there exists an essential annulus
in $N$ represented by an embedding $g: (S^1 \times I, \ S^1 \times
\partial I) \rightarrow (N, \partial N)$.}

The existence of an embedded annulus $A \subset M_2$ implies the
existence of an essential torus $T_3$ in
$M_2$, which is {\it not} parallel to the boundary. $T_3$ is constructed as follows: Let $C_i \subset \partial
M_2$, $i=1,2$ be the boundary
components (which are circles) of $A$ and let  $\nu$ be a sufficiently
small open tubular neighborhood
of $A_0$ where $A_0 = A \setminus \partial A$. If the boundary tori of $M_2$ containing $C_i$ are $T_i$ then
the required essential torus
is $T_3=(T_1 \setminus (T_1 \cap \bar \nu) \cup \partial \nu \cup T_2 \setminus
(T_2 \cap \bar \nu))$ where $\bar \nu$ is the closure of $\nu$. Since $T_3$ is  not parallel to the boundary,
it follows from the
Torus Decomposition of 3-manifolds (at the beginning of the present subsection)
that $M_2$ is Seifert-fibered.


Hence, if $g$ fixes three consecutive edges $[v_0,v_1]$,  $[v_1,v_2]$,
$[v_2,v_3]$ of $T$, then the pieces
corresponding to $G_{v_1}$ and  $G_{v_2}$ must be Seifert-fibered.

One then observes that any
essential annulus $A$  in a
Seifert-fibered 3-manifold $N$
represents an element of the center $Z(\pi_1(N))$.
To see this, note that
the presence of an essential annulus implies that the fundamental groups of the corresponding boundary tori
 intersect. On the other hand, a Seifert fibered space with
boundary is a trivial circle bundle over an orbifold and is in fact finitely covered by the
product of a surface and a circle
\cite{scott-geoms}. Hence the intersection of  fundamental groups of the two boundary tori
is an infinite cyclic group
generated by the vertical circle (unless the base orbifold is itself a cylinder $S^1 \times I$).
In particular the annulus represents an element of the center of the fundamental group of the Seifert fibered space $N$.


Therefore
$g \in Z(G_{v_1})\cap Z(G_{v_2})$, the intersection of the centers of
$G_{v_1}$ and  $G_{v_2}$. But $Z(G_{v_1})\cap Z(G_{v_2})= \{1 \}$  because adjacent Seifert-fibered pieces in the torus decomposition cannot have homotopic fibers (when each piece is regarded as a trivial circle bundle over an orbifold). Hence $g=1$. 
\end{proof}

\subsection{Coboundedness and Finite Height} In this subsection, we shall give an alternate approach to 
the content of Proposition \ref{remark-torus}.
The purpose of this subsection is to show that if $(\pi_1(M),\pi_1(M_1))$
is  a
 pair such that $M$ is a prime 3-manifold admitting a non-trivial torus decomposition and $M_1$ is one of the pieces, then 
{\bf $\pi_1(M_1)$
has finite height in $\pi_1(M)$}. In particular,
the pair $(\pi_1(M),\pi_1(M_1))$ satisfies the hypotheses on the pair $(Q,R)$ occurring in Corollary \ref{pppcor}. It follows that the action of
such a 
$\pi_1(M)$ on the Bass-Serre tree associated to the splitting over $\pi_1(M_1)$ is faithful and minimal.

\smallskip

\noindent {\bf Structure of Ends:} Let $L_1$ be a non-compact
hyperbolic 3-manifold of finite volume, and
let $L$ be the manifold with boundary obtained from $L_1$ by removing the interiors of cusps.
Let $H_0$ be the torus subgroup of $\pi_1(L)$ corresponding to a cusp. Let $L_T$ denote the cover of $L$
corresponding to the torus subgroup $H_0$. Then $L_T$ is
non-amenable; an easy way to see this is to use the fact that $\pi_1(L)
$ is strongly hyperbolic relative to the collection of cusp subgroups
\cite{farb-relhyp}.

Similarly, if $\Sigma$ is a compact 2-manifold with boundary and genus greater than one, let $H_0$
denote the cyclic subgroup of $\pi_1( \Sigma )$ corresponding to a boundary component. Then
$\pi_1( \Sigma )$ is strongly hyperbolic relative to $H_0$ (cf. Proposition 2.10 of \cite{mahan-agt}). Hence the cover of
$\Sigma$ corresponding to $H_0$ is non-amenable. This implies further that the cover of $\Sigma \times S^1$
corresponding to $H_0 \oplus \ints$ is non-amenable.

Let $M$ be a prime 3-manifold admitting a non-trivial torus decomposition.  By passing to a finite cover, if necessary, and  using Theorem \ref{new}, we can assume that each piece
of the torus decomposition of $M$ is either a non-compact complete hyperbolic manifold of finite volume or of the form $\Sigma \times S^1$, with $\Sigma$ a compact surface with non-empty boundary and
genus at least two.

Let $M_1$ be a piece of the torus decomposition.
Note that each boundary component of $M_1$ is a torus. Then the cover of $M$ corresponding to the subgroup $\pi_1(M_1) \subset
\pi_1(M)$ has ends corresponding to the different torus boundary components of $M_1$. Since each end of this cover is
non-amenable by the above discussion we have the following lemma.

\begin{lemma} \label{nonam} Let $M$ be a prime 3-manifold admitting a non-trivial torus decomposition. Let $M_1$ be a piece of the torus decomposition.
Then the cover of $M$ corresponding to the subgroup $\pi_1(M_1) \subset
\pi_1(M)$  is
non-amenable. \end{lemma}

\noindent {\bf Coboundedness:} \vspace{2mm}

\begin{defn}
Let $(Y,d)$ be a  metric space. We say that two subsets $Y_1, Y_2
\subset Y$ are mutually cobounded if for all $D \geq 0$, the subset
$$\{ (y_1, y_2) \in Y_1 \times Y_2 \,\mid\, d(y_1, y_2) \leq D \}$$
is compact. \end{defn}

See Section 1 of \cite{mahan-mbdl} for a closely  related
definition in the context of hyperbolic metric spaces.

The next two lemmas give standard examples of  cobounded subsets of
hyperbolic metric spaces (see \cite{farb-relhyp}, \cite{bowditch-relhyp}
or \cite[Section 2.2]{mahan-split} for proofs).

 We note from Lemma \ref{rtc} below that these examples give groups $Q$ and subgroups $R$ such that no infinite normal subgroup of $Q$ is contained in $R$.

\begin{lemma} Let $M$ be a non-compact complete hyperbolic $n$-manifold
of finite volume. Let $H_1$ and $H_2$ be lifts of a
neighborhood of a cusp in $M$ to the universal cover $\widetilde{M}$. Then $H_1$ and $H_2$ are cobounded horoballs
in $\widetilde{M}$. Equivalently, $\partial H_1$ and $\partial H_2$
are cobounded horospheres
in $\widetilde{M}$. \label{horo-cob} \end{lemma}

\begin{lemma} Let $\Sigma$ be a hyperbolic surface with (possibly empty) totally geodesic boundary. Let $\sigma_1$ and $\sigma_2$ be lifts of a
closed geodesic in $\Sigma$ to the universal cover $\widetilde{\Sigma}$. Then $\sigma_1$ and $\sigma_2$ are cobounded
in $\widetilde{\Sigma}$. \label{geod-cob} \end{lemma}

Now let $N = \Sigma \times S^1$, where $\Sigma$ is an orientable surface with boundary, having genus greater than one. Equip $N$
with a product metric. It follows
from Lemma \ref{geod-cob}
that any two lifts $\sigma_1$ and $\sigma_2$ of the boundary components
of $\Sigma$ to the universal cover $\widetilde{N}$ are cobounded.
Let $E_i = \sigma_i \times \rls$ denote the corresponding lift of bounding tori of $N$
 to $\widetilde{N}$ containing $\sigma_i$.
Thus we have the following:

\begin{corollary} Let $N = \Sigma \times S^1$, where $\Sigma$ is an
orientable
surface with boundary, having genus greater than one.  Equip $N$
with a product Riemannian metric, and denote by $d_{\widetilde{\Sigma}}$
the induced distance function on  $\widetilde{\Sigma}$.
Let $\sigma_1$ and $\sigma_2$ be lifts of the boundary components of
$\Sigma$ to the universal cover $\widetilde{N}$.
Let $E_i = \sigma_i \times \rls$, $i\,=\, 1\, ,2$, denote the
corresponding lifts of bounding tori of $N$
 to $\widetilde{N}$ containing $\sigma_i$.  Let $p_i \in
\sigma_i$ be such that $$d_{\widetilde{\Sigma}} (p_1, p_2)
= d_{\widetilde{\Sigma}} (\sigma_1, \sigma_2)\, .$$ Then for all $D \geq
0$, there exists $\epsilon > 0$ such that
$$\{ (y_1, y_2) \in E_1 \times E_2 | d(y_1, y_2) \leq D \} \subset
(B_\epsilon (p_1) \times \rls) \times
(B_\epsilon (p_2) \times \rls)\, ,$$ where $B_\epsilon (p_i)$ denotes
the $\epsilon$-neighborhood of $p_i$ in $\sigma_i$. \label{critcor}
\end{corollary}

Let, as at the beginning of this subsection,
$M$ be a prime 3-manifold admitting a non-trivial torus decomposition.
Let $M_1$ be a piece of the torus decomposition.
Let $\widetilde{M}$ denote the universal cover of $M$ equipped with some equivariant Riemannian metric.
 The bounding tori of $M_1$ lift to copies  of the Euclidean plane
in  $\widetilde{M}$.
 Let $\{ \E_\alpha \}$ be the collection of lifts (to $\widetilde{M}$) of all tori in the torus decomposition of $M$. If $M_1$ is a hyperbolic piece, then by
Lemma \ref{horo-cob}, it follows that any pair $E_1, E_2 \in \{ \E_\alpha \}$  which are lifts of bounding tori of $M_1$ is cobounded in $\widetilde{M}$.
In fact since the lifts of any hyperbolic piece of the torus decomposition separate $\widetilde{M}$, it follows that there exist
$E_1, E_2 \in \{ \E_\alpha \}$ separated by such a lift. Hence we have:

\begin{lemma} Let $M$ be a prime 3-manifold admitting a non-trivial torus decomposition such that at least one
of the pieces of the torus decomposition is hyperbolic.
Let $\{ \E_\alpha \}$ be the collection of lifts (to $\widetilde{M}$) of all tori in the torus decomposition of $M$. There exist $E_1, E_2 \in \{ \E_\alpha \}$
which are cobounded in $\widetilde{M}$. \label{hypt-cob} \end{lemma}

Next assume that $M$ is a prime 3-manifold admitting a non-trivial torus decomposition such that
 all pieces of $M$ are Seifert-fibered, i.e., $M$ is a graph manifold. By passing to a finite-sheeted cover if necessary
we may assume, by Theorem \ref{new}, that each piece of the torus decomposition  is of the form $\Sigma \times S^1$, where $\Sigma$ is a compact surface
with boundary and genus greater than one.

In the universal cover $\widetilde{M}$, let
$\widetilde{M}_1$, $\widetilde{M}_2$, $\widetilde{M}_3$,
$\widetilde{M}_4$ be a sequence of lifts of Seifert-fibered pieces such
that $\widetilde{M}_i \bigcap \widetilde{M}_{i+1} = E_i$, for $i = 1,
2, 3$, where
$E_i$ is the universal cover of a torus $T_i$ in $M$. We shall show
that $E_1$ and $E_3$ satisfy the conclusions of Lemma \ref{hypt-cob}.

Let  $M_i = \Sigma_i \times \alpha_i$ where $\alpha_i$ is
a unit circle. Assume without loss of generality (by passing to finite-sheeted covers if necessary)
that $\{M_i\}_{i=1}^4$ are embedded submanifolds of $M$ and that ${M_i}
\bigcap{M_{i+1}} = T_i$, for $i = 1, 2, 3$, where
$T_i$ is an embedded essential torus in ${M}$ appearing in the torus decomposition.
Also,  $T_i = \sigma_i \times \alpha_i$ where $\sigma_i$ is a boundary curve of $\Sigma_i $ and $\alpha_i$ is the circle fiber.

By Corollary \ref{critcor} the set
of points $x \in E_1$ and $y \in E_2$ with $d(x,y) \leq D$ must lie
within $(B_\epsilon (p_1) \times \widetilde{\alpha_2}) \times
(B_\epsilon (p_2) \times \widetilde{\alpha_2})$, where $[p_1,p_2]$
denotes the geodesic in $\widetilde{\Sigma}_2$ joining
the appropriate lifts of $ \sigma_2$.  The same argument shows that
the set
of points $u \in E_2$ and $v \in E_3$ with $d(u,v) \leq D$ must lie
within $(B_\epsilon (q_1) \times \widetilde{\alpha_3}) \times
(B_\epsilon (q_2) \times \widetilde{\alpha_3})$, where $[q_1,q_2]$
denotes the geodesic in $\widetilde{\Sigma_3}$ joining
the appropriate lifts of $\sigma_3$. Next,
$(B_\epsilon (p_2) \times \widetilde{\alpha_2})$ and $(B_\epsilon (q_1)
\times \widetilde{\alpha_3}) $
are cobounded as the fibers of $M_2$ and $M_3$ do not agree, being
different Seifert components. Hence we conclude that
$E_1$ and $E_3$ are cobounded in $\widetilde{M}$, in other words,
$E_1$ and $E_3$ satisfy the conclusions of Lemma \ref{hypt-cob} as
claimed above.

Combining this with
Lemma \ref{hypt-cob}, we conclude:

\begin{prop} Let $M$ be a prime 3-manifold admitting a non-trivial torus decomposition. Let $\{ \E_\alpha \}$ be the collection of lifts (to $\widetilde{M}$) of all tori in the torus decomposition of $M$. There exist $E_1, E_2 \in \{ \E_\alpha \}$
which are cobounded in $\widetilde{M}$.  \label{cob} \end{prop}

In fact we have shown the following, which is the main Proposition of this subsection and might be of independent interest.

\begin{prop}  Let $M$ be a prime 3-manifold admitting a non-trivial torus decomposition and let $M_1$
be one of the pieces of the decomposition. Then $\pi_1(M_1)$ has finite height in $\pi_1(M)$. \label{remark-torus1}
\end{prop}

\medskip

Infinite normal subgroups offer the opposite situation, as shown by the
following lemma.

\begin{lemma}\label{rtc}
Let $M$ be a topological space and $A \subset M$ an incompressible
subspace (i.e., $i: A \longrightarrow M$ induces
an injective map $i_\ast : \pi_1(A) \longrightarrow \pi_1(M)$). Suppose
$N$
is an infinite normal subgroup of $\pi_1(M)$ contained in $\pi_1(A)$.
If $\widetilde A$ is a lift of $A$ to $\widetilde M$ (the universal cover of $M$), then $g \widetilde A$ and $h \widetilde A$ are never cobounded for any $g, h \in \pi_1(M)$, where we regard $\pi_1(M)$ as the group of deck transformations of $\widetilde M$ and $\widetilde M$ has been endowed with a $\pi_1(M)$-equivariant distance function $d$.
 \label{ncob} \end{lemma}

\begin{proof}
It is enough to prove that $\widetilde A$ and $g \widetilde A$ are not cobounded for any $g \in \pi_1(M)$.
Let $\{t_n \}_{n=1}^\infty \subset N$ be an infinite subset and let $s_n:=gt_ng^{-1} \in N$. Fix $x \in \widetilde A$. Then
$s_nx \in \widetilde A, \ gt_nx \in \widetilde A$ and $d(s_nx, \ gt_nx) =d(x,gx)$. Since $\{gt_nx\}$ is a noncompact subset of $\widetilde M$, we see that
$\widetilde A$ and $g \widetilde A$ are not cobounded by taking
$D=d(x,gx)$.
\end{proof}

\subsection{K\"ahler Groups and 3-Manifold Groups}

In this Section we use the restrictions obtained above to rule
out various possibilities. In this section, the following possibilities
will be taken up and ruled out one-by-one.\\
(a) $M$ admits a non-trivial prime decomposition. \\
(b) $M$ admits a non-trivial torus decomposition. \\
(c) $M$ admits a $Sol$ geometric structure.\\
(d) $M$ admits a hyperbolic structure.\\

Spherical or elliptic $3$-manifolds have finite fundamental group and $3$-manifolds with $S^2 \times \rls$ geometry have virtually
cyclic fundamental group. These are ruled out by the hypothesis on $Q$.
Thus at the end of the discussion we shall conclude that $M$
is a Seifert-fibered space with Euclidean or hyperbolic base orbifold. We shall then proceed to extract further restrictions on
what $Q$ may be in case $M$ is Seifert-fibered.

\smallskip

We will use  the following proposition for (a) and (b):

We will also need the following simple lemma:

\begin{lemma}\label{plm}
If
$$1 \longrightarrow N_0 \longrightarrow \pi_1(S) \longrightarrow Q
\longrightarrow
1$$ is an exact sequence of finitely generated groups
such that $S$ is a closed orientable surface and $Q$ is an infinite non-virtually cyclic fundamental group of a 3-manifold,
then $N_0$ cannot be finitely generated. \label{impos} \end{lemma}

\begin{proof}
If $S$ has genus $g > 1$ then the only finitely generated normal subgroup of such a $\pi_1(S)$
is the trivial group. This forces $Q$ and $\pi_1(S)$ to be
commensurable, an impossibility.

 For $S$ a torus, $\pi_1(S) = \ints \oplus \ints$ and again, such an
exact sequence is impossible as it forces
 $Q$ and $\pi_1(S)$ to be commensurable.
\end{proof}

\noindent {\bf $M$ admits a non-trivial prime decomposition}: This follows from Proposition \ref{ppp} and Lemma \ref{plm} by taking $Q=\pi_1(M)$ and $R=\{1\}$.

\smallskip

\noindent {\bf $M$ is prime and admits a non-trivial torus decomposition}:  This follows from Proposition \ref{remark-torus} 
(or from Proposition \ref{remark-torus1}) and
Corollary \ref{pppcor}.

\vspace{3mm}

\noindent {\bf $M$ is Sol} \vspace{2mm}

Let $G$ be as in ($\ast$) such that $Q$ is the fundamental group of
a 3-manifold $M$ which is Sol.

Since the group $G$ admits a surjection to the solvable non-nilpotent group
$\pi_1(M)$, Theorem \ref{brudnyi} shows that $G$ cannot be K\"ahler.

\vspace{3mm}

\noindent {\bf $M$  hyperbolic} \vspace{2mm}

Let $G$ be as in ($\ast$) such that $Q$ is the fundamental group of
a 3-manifold $M$ which is hyperbolic.

Since the group $G$ admits a surjection  to a closed hyperbolic
3-manifold group $Q$, $G$ cannot be K\"ahler by Theorem \ref{dp}.

\medskip

By the above discussion, it follows that
if $$ 1 \longrightarrow N \stackrel{i}{\longrightarrow} G
\stackrel{q}{\longrightarrow} Q \longrightarrow 1$$  is an exact
sequence of
groups such that \\
(a) $Q = \pi_1(M)$ is
the fundamental group of a closed 3-manifold $M$,\\
(b) $Q$ is infinite and not virtually cyclic, \\
(c) $G$ is  K\"ahler, and \\
(d) $N$ is finitely presented, \\
then $M$ is Seifert-fibered with base orbifold Euclidean or hyperbolic.

\vspace{3mm}

\noindent {\bf $M$  Seifert-fibered} \vspace{2mm}

Let $G$ be as in ($\ast$) such that $Q$ is the fundamental group of
a 3-manifold $M$ which is Seifert-fibered.

Given a surjective homomorphism $h: G \longrightarrow G_1$ with kernel
$K$, there is an Euler class
obstruction $e(h)$ to the existence of a section of $h$ associated to the exact sequence
$$ 1 \longrightarrow \frac{K}{[K,K]} \longrightarrow \frac{G}{[K,K]}
\longrightarrow G_1 \longrightarrow 1\, . $$
A recent Theorem of Arapura (Corollary 5.5 of \cite{arapura-hom}) asserts the following.

\begin{theorem} If a K\"ahler group $G$ admits a surjective
homomorphism $h$ to a surface group $G_1=\pi_1(S)$ of genus $g$ greater than one
with $g$  maximal, then the Euler class $e(h) \in H^2(G_1,
\frac{K}{[K,K]})$ is torsion, where $K$ is the kernel of $h$.
\label{torsion} \end{theorem}

For a 3-manifold $M$ which is a twisted (non-zero Euler class) circle bundle over a closed hyperbolic orbifold,  some
 cover is  a twisted  circle bundle over a closed surface $S$ of  genus greater than one. By abusing notation slightly
we call this cover $M$ and its fundamental group $Q$. Let $G_1=\pi_1(S)$. If
$$1 \,\longrightarrow\, N_0 \,(= \ints )\,  \longrightarrow \,Q\,
\stackrel{\phi}{\longrightarrow} \,G_1\, \longrightarrow \,1$$ is the
associated exact sequence, then
$e(\phi )$ is non-torsion.  We have  the following.

\begin{prop} Let $$ 1 \longrightarrow N \stackrel{i}{\longrightarrow} G
\stackrel{q}{\longrightarrow} Q \longrightarrow 1$$  be an exact
sequence of
groups such that
\begin{enumerate}
\item[(a)] $Q = \pi_1(M)$ is the fundamental group of a Seifert-fibered
closed 3-manifold $M$ with hyperbolic base-orbifold,
\item[(b)] $G$ is a K\"ahler  group, and
\item[(c)] $N$ is finitely presented.
\end{enumerate}
Then a finite-sheeted cover of $M$ is a product $\Sigma \times S^1$,
where $\Sigma$ is a closed oriented surface of genus greater than one.
\label{sfskahler} \end{prop}

\begin{proof}
If $G$ is K\"ahler, any finite index subgroup of $G$ is K\"ahler.  Let $G= \pi_1(X)$, where
$X$ is a K\"ahler manifold. By passing to a finite-sheeted cover of $X$ if necessary,
we may assume that $M$ is a circle
bundle over a closed surface $S$ of  genus greater than one. Abusing notation slightly again, call this cover $X$.
We have exact sequences
$$1 \,\longrightarrow \,N \,\stackrel{i}{\longrightarrow}\, \pi_1(X)\,
\stackrel{q}{\longrightarrow} \,Q \,(=\pi_1(M))\,\longrightarrow\, 1
$$
and
$$ 1 \,\longrightarrow \,N_0\, (= \ints )\,  \longrightarrow \,Q\,
\stackrel{\phi}{\longrightarrow} \,G_1 \,(=\pi_1(S))\, \longrightarrow
\, 1\, .
$$

This gives rise to a surjection $\phi \circ q: \pi_1(X) \longrightarrow
G_1 (=\pi_1(S))$. If $e(\phi )$ is non-zero, it follows that
$e(\phi \circ q )$ is non-torsion. Therefore $e(\phi )=0$
by Theorem \ref{torsion}.
\end{proof}

Hence we have the following, which is one of the main theorems of our paper.

\begin{theorem}
Let $$ 1 \longrightarrow N \stackrel{i}{\longrightarrow} G
\stackrel{q}{\longrightarrow} Q \longrightarrow 1$$  be an exact
sequence of
groups such that
\begin{enumerate}
\item[(a)] $Q = \pi_1(M)$ is the fundamental group of a  closed
3-manifold $M$ such that $Q$ is infinite and not virtually cyclic,
\item[(b)]  $G$ is a K\"ahler  group, and
\item[(c)]  $N$ is finitely presented.
\end{enumerate}
Then there exists a finite index subgroup $Q^\prime$ of $Q$ such that
either
$Q^\prime$ is a finite index subgroup of the $3$-dimensional
Heisenberg group or
 $Q^\prime = \pi_1(\Sigma) \times S^1$, where $\Sigma$ is a closed
oriented surface of positive genus.
\label{main1} \end{theorem}

\begin{proof}
Proposition \ref{sfskahler} deals with the case that the base orbifold is hyperbolic. If the
base orbifold is Euclidean, then any twisted bundle over the torus has fundamental group isomorphic to
(a finite index subgroup of) the $3$-dimensional Heisenberg group.
\end{proof}

If $N$ is trivial and $G$ is a K\"ahler group, then the exact sequence
$(\ast )$ implies that $Q$ is
a K\"ahler group. But $b_1(\pi_1(\Sigma \times S^1))$ is odd for
any closed oriented surface $\Sigma$, hence $\pi_1(\Sigma \times S^1)$
is not K\"ahler. Also, no finite index subgroup
of the $3$-dimensional Heisenberg group can be K\"ahler (Example 3.31 in
\cite[p. 40]{abckt}). Hence, in view of Theorem \ref{main1}, we have a
new proof of the following theorem of Dimca-Suciu
(conjectured originally by Donaldson and Goldman).

\begin{theorem}[\cite{ds}, \cite{kotschick}]\label{dsk}
Let $Q$ be the fundamental group of a closed 3-manifold. Suppose
further  that $Q$ is infinite and not virtually cyclic.
Then $Q$ cannot  be a K\"ahler group. \end{theorem}

\section{Complex Surfaces}

\subsection{Restrictions for Complex Non-K\"ahler Surfaces}
We start with the following theorem  due to Kodaira.

\begin{theorem}[See Theorem 1.28, Corollary 1.29 of \cite{abckt})]
Let $\Gamma$ be a finitely presented group.
Then $\Gamma$ is the fundamental group of a K\"ahler surface if and
only if it is the fundamental group of a  compact complex  surface with
even first Betti number. \label{kodaira-Kahler} \end{theorem}

We collect together known
restrictions on fundamental groups of compact non-K\"ahler complex surfaces that we will need.

The following Theorem is due to Kodaira; see Theorems 1.27, 1.38 of
\cite{abckt}, \cite[pp. 244--245]{bhpv}, \cite[Ch.~2]{fm}):

\begin{theorem}
Let $X$ be a compact complex  surface with odd first Betti number. Then
either $X$ is elliptic or of Class VII, i.e.,
$b_1 (X)=1$ and the Kodaira dimension of $X$ is negative.

If $X$ is minimal of Class VII, and admits non-constant
meromorphic functions, then $X$ is diffeomorphic to $S^3 \times S^1$.

If $X$ is a minimal elliptic surface with odd first Betti number, then
the Euler characteristic $\chi (X) = 0$ and the universal
cover $\widetilde{X}$ is diffeomorphic to $S^3 \times {\mathbb{R}}$
or ${\mathbb{R}}^4$.
 \label{kodaira-cx}
\end{theorem}

\begin{corollary}[Corollary 1.41 of \cite{abckt}] Let $G$ be a group
with
$b_1(G) = 2m+1$ with $m\geq 1$ such that $G$ is the fundamental group
of a complex surface $X$, Then $X$ is elliptic. If the elliptic fibration
$X$ is nonsingular, then $X$ is a $K(G,1)$ space. \label{ell}
\end{corollary}

The next few statements impose restrictions on fundamental groups of class VII complex surfaces.

\begin{lemma}[Lemma 1.45 of \cite{abckt}]
If $X$ is a minimal complex surface of Class VII, then any finite-sheeted cover of $X$ is of class VII.
Further the intersection form on $H_2(X, \mathbb{Z})/{\rm Tor}$ is
negative definite.  \label{vii} \end{lemma}

\begin{corollary} Let $G$ be a finitely presented group. If $G$ has a
finite index subgroup $G^\prime$ such that
$b_1 (G^\prime ) > 1$, then $G$ cannot be the fundamental group of a
complex surface of Class VII. \label{viicor} \end{corollary}

Combining Theorem
\ref{luecke} with Corollary \ref{viicor}, we have the following:

\begin{corollary}
Let $$ 1 \longrightarrow N \stackrel{i}{\longrightarrow} G
\stackrel{q}{\longrightarrow} Q \longrightarrow 1$$  be an exact
sequence of
groups, where $Q$ is
the fundamental group of a closed 3-manifold $M$, and $N$ is finitely
presented. Suppose further that at least
one of the following conditions is satisfied:
\begin{enumerate}
\item[(a)] $M$ contains  an incompressible torus and is not a $Sol$
manifold.
\item[(b)] $M$ is a Seifert-fibered space such that the base of the
Seifert fibration is not an elliptic orbifold.
\item[(c)]  $M$ is not prime.
\end{enumerate}
Then $G$ cannot be the fundamental group of a complex surface of class
VII. \label{not7} \end{corollary}

\begin{proof}
This is a consequence of the following simple fact: Let
$A$ and $B$ be topological spaces
and $q: \pi_1(A) \longrightarrow \pi_1(B)$ a surjective homomorphism. If
$B_1$ is a finite cover of $B$ with positive $b_1$, then the finite
cover $A_1$ of $A$ corresponding to the
subgroup $q^{-1}(B_1)$ also has positive $b_1$.
\end{proof}

\begin{theorem}[\cite{carlson-toledo-vii}]  Let $M$ be a non-elliptic
compact complex surface with first Betti number one and admitting no
nonconstant meromorphic functions. Let $N$ be
a compact Riemannian manifold of constant negative curvature. Let $\phi
: \pi_1 (M) \longrightarrow \pi_1 (N)$
be a homomorphism. Then the image of $\phi$ is either trivial or an infinite cyclic group. \label{ctvii} \end{theorem}

\subsection{Class VII Surfaces}

Let $$ 1 \longrightarrow N \stackrel{i}{\longrightarrow} G
\stackrel{q}{\longrightarrow} Q \longrightarrow 1$$  be an exact
sequence of
groups where $Q$ is an infinite, non-virtually cyclic
 fundamental group of a closed 3-manifold $M$.  Suppose further that $G$ is the fundamental group of a Class VII surface. {\it No assumptions are made on
$N$ for the purposes of this section.}
Then by Corollary \ref{not7} $M$ is prime and does not admit a non-trivial torus decomposition. Also $M$ cannot be Seifert fibered  over a hyperbolic or flat orbifold.

One can assume that $X$ is minimal (as blowing down does not change fundamental
group).
Then either $X$ is a Hopf surface or else it does not admit any non-constant meromorphic functions.

$X$ cannot be Hopf as Hopf surfaces have infinite cyclic fundamental group
and $\pi_1(M)$ is infinite, and not virtually cyclic. Hence $X$ does not admit any non-constant meromorphic functions.

  The quotient map from $G$ to $Q$ is surjective. Hence, by
Theorem \ref{ctvii}, the manifold $M$ cannot be hyperbolic.

Finally we dispose of the case that $Q$ is $Sol$.
This
 is ruled out by the following adaptation of an argument of
Kotschick \cite{kotschick} that we reproduce here.

Since $\pi_1(X)$ surjects onto $\pi_1(M)$ and the latter fibers over the circle,
it follows that the classifying map $\phi_q$ for $q: \pi_1(X)
\longrightarrow \pi_1(M)$ induces an isomorphism
$$\phi_{q,1}^\ast : H^1(M) \longrightarrow H^1(X)$$ and an injective map
$\phi_{q,2}^\ast : H^2(M) \longrightarrow H^2(X)$. Let $\alpha$ be a
generator of $H^2(M, {\mathbb{Z}})$. Then
$\alpha \cup \alpha = 0$ and so $\phi_{q,2}^\ast \alpha \cup \phi_{q,2}^\ast \alpha = 0$.
 which makes the intersection form indefinite.
This contradicts Lemma \ref{vii}.
We summarize our conclusions as follows.

\begin{theorem}
Let $$ 1 \longrightarrow N \stackrel{i}{\longrightarrow} G
\stackrel{q}{\longrightarrow} Q \longrightarrow 1$$  be an exact
sequence of
groups such that
 $Q = \pi_1(M)$ is
the fundamental group of a  closed 3-manifold $M$ which is infinite and not virtually cyclic.
Then $G$ cannot be the fundamental group of a Class VII complex surface.
\label{main7} \end{theorem}

\subsection{Elliptic Fibrations}
Let $$ 1 \longrightarrow N \stackrel{i}{\longrightarrow} G
\stackrel{q}{\longrightarrow} Q \longrightarrow 1$$  be an exact
sequence of
groups where $Q$ is an infinite, non-virtually cyclic
 fundamental group of a closed 3-manifold $M$.  Suppose further that $G$ is the fundamental group of a compact complex surface $X$
admitting an elliptic fibration. Assume that $N$ is finitely generated.

If the fibration is singular, then the inclusion of the fiber subgroup $\ints \oplus \ints$
into $\pi_1(X)$ has non-trivial kernel and
 $\pi_1(X)$ must have rational cohomological dimension at most 3. In fact, by Theorem 2.3
of \cite{fm},
existence of singular fibers forces $\pi_1(X)$ to be equal to the fundamental group of the base orbifold of complex dimension one.
Hence we have an exact sequence
$$ 1 \longrightarrow N \stackrel{i}{\longrightarrow} G
\stackrel{q}{\longrightarrow} Q \longrightarrow 1$$ where $G$ is the
fundamental
group of an orbifold of complex dimension one and $N$ is a finitely generated
normal subgroup of $G$. If the orbifold is hyperbolic, the finitely generated normal subgroup
$N$ must be finite and $Q$ must therefore have rational cohomological dimension two.
But the assumptions on $Q$ force it to have cohomological dimension 3.

If the orbifold is Euclidean, then after passing to a finite index subgroup if necessary, we
have an exact sequence
$$ 1 \longrightarrow N \stackrel{i}{\longrightarrow} \ints \oplus \ints
\stackrel{q}{\longrightarrow} Q \longrightarrow 1\, .$$ This forces $Q$
to
have
cohomological dimension at most two. A contradiction again.

Hence the elliptic fibration must be non-singular (with possibly multiple fibers)
and we have an exact sequence (after passing to a finite cover again if necessary)
\begin{equation}\label{ef3}
1 \,\longrightarrow \,N_0 \,(=\,  \ints \oplus \ints )\,
\stackrel{\phi}{\longrightarrow} \,G\, \stackrel{\psi}{\longrightarrow}
\,\pi_1(S) \,\longrightarrow\, 1\, ,
\end{equation}
where $S$ is a
closed orientable surface of positive genus. In particular
by Corollary \ref{ell}, the group $G$ admits a closed aspherical
4-manifold
$X$ as a $K(G,1)$. Further, there is a holomorphic fiber bundle
structure on $X$ realizing the exact sequence \eqref{ef3} with one
dimensional
complex tori  as fibers. The universal cover of $X$ is homeomorphic to
$\rls^4$.

We first show that $M$ has to be prime. Assume that $M$ is not
prime. Then $Q$ is a
non-trivial free product that is not virtually
cyclic. Hence the abelian normal subgroup $q(N_0) \subset Q$ must be
trivial, i.e., $N_0 = {\mathbb{Z}} \oplus {\mathbb{Z}}$
 must be contained in $N$. Therefore we have an exact sequence
$$1 \longrightarrow  \frac{N}{N_0} \longrightarrow \pi_1(S)
\longrightarrow Q \longrightarrow 1\, ,$$
where $\frac{N}{N_0}$ is a finitely generated normal subgroup of
$\pi_1(S)$, and $Q$ is a non-trivial free product.
This is again impossible  by Lemma \ref{impos}. Hence $M$ is prime.

Since $\pi_1(M)$ is not virtually cyclic, it follows that $Q$ is a $PD(3)$ group  by the sphere theorem.
We shall need the following Theorem of Bieri and Eckmann.

\begin{theorem}[\cite{be}]
Let $1 \longrightarrow K \longrightarrow H \longrightarrow
L\longrightarrow 1$  be a
short exact sequence of groups.
If $K$ and $L$ are duality groups of dimensions $n$ and $m$ respectively, then $H$ is
a duality group of dimension $(m+n)$.  \label{duality} \end{theorem}

We shall try to understand the structure of $M$ in terms of the rank of
$q(N_0)$. The rank of $q(N_0)$ is zero, one, or two.\\

\smallskip

\noindent {\bf Case 1: Rank of $q(N_0)$ is  zero.} Then $N \bigcap N_0$
has rank 2 and we have an exact sequence
$$1 \longrightarrow N \bigcap N_0 \longrightarrow G^\prime
\stackrel{q}{\longrightarrow} Q \longrightarrow 1\, ,$$ where
$G^{\prime}$
is a
finitely presented subgroup of $G$. The exact sequence
 forces $G^\prime$ to have rational cohomological dimension $5$ by Theorem \ref{duality}. Since $G$ has
cohomological dimension $4$, this is impossible.

\smallskip

\noindent {\bf Case 2: Rank of $q(N_0)$ is  one.} Since
the rank of $q(N_0)$ is  one, then $Q$ has a normal $\ints $ subgroup. Hence $M$ is a Seifert-fibered space by by Theorem \ref{sfs}.
Further $N \bigcap N_0$ has rank one and so there exists an infinite
cyclic
 subgroup $N_1$ of finite index in $N \bigcap N_0$. Hence we have an
exact sequence
$$1 \longrightarrow N_1 \longrightarrow G^\prime
\stackrel{q}{\longrightarrow} Q \longrightarrow 1\, ,$$ for some
subgroup
$G^\prime$ of $G$. This forces $G^\prime$ to be
a Poincar\'e duality group of dimension $4$ by Theorem \ref{duality} and so $G^\prime$ is of finite index in $G$. Hence a finite-sheeted cover of $X$ is a
circle bundle over a Seifert-fibered space.

\smallskip

\noindent {\bf Case 3: Rank of $q(N_0)$ is  two.} Then $Q$ has a normal $\ints \oplus \ints$ subgroup. Hence $M$ is virtually a torus bundle over
the circle \cite{hempel-book}. Further, $N \bigcap N_0 = \{1 \}$. Hence
$NN_0$ is a normal subgroup
of $G$ and for all $n \in N, m \in N_0, mn=nm$. Also, $\frac{G}{NN_0} =
\ints$.
Since $G$ is torsion-free, so is $N$. Also, $N$ must be infinite
as $G$ has rational cohomological dimension $4$ and $Q$ has rational
cohomological dimension $3$. Let $H$ be any infinite cyclic
subgroup of $N$. Then $HN_0$ is isomorphic to $\ints \oplus \ints \oplus \ints$ and we have a short exact sequence
\begin{equation}\label{ef4}
1 \longrightarrow HN_0 \longrightarrow  G^\prime
\longrightarrow H_1 (=\ints ) \longrightarrow 1
\end{equation}
for some subgroup $G^\prime$ of $G$. This forces $G^\prime$ to
be
a Poincar\'e duality group of dimension $4$ by Theorem \ref{duality} and so $G^\prime$ is of finite index in $G$.
Further it forces $H$ to have finite index in $N$ and hence $N$ is
infinite cyclic. Also the exact sequence in \eqref{ef4} must split
and hence by passing to a double cover if necessary we may assume that
if $s: H_1 \longrightarrow G^\prime$ is a section,
then for all $h \in H$ and $h_1 \in s(H_1)$, $hh_1=h_1h$. Therefore, $X$
is virtually both a $3$-torus bundle over the circle as well as
a $2$-torus bundle over a $2$-torus. By passing to a finite-sheeted cover if necessary, we may assume that $N$ is central.
Since each element of $N$ commutes with each element of $N_0$ as well as each element of $s(H_1)$, it follows that
$G^\prime = H \oplus Q^\prime$, where $ Q^\prime$ is the fundamental group of a $2$-torus bundle over the circle.

If $M$ has $Sol$ geometry, then $ Q^\prime$ is a $Sol$ group also and hence $b_1(G^\prime )= 1 + b_1(Q^\prime) = 2$, which in turn implies that
the cover of $X$ with fundamental group $G^\prime$ is K\"ahler. This contradicts Theorem \ref{main1}. Let $M^\prime$ be the
cover of $M$ corresponding to $ Q^\prime$. Hence the monodromy of the torus bundle $M^\prime$ is either reducible or periodic.
In either case, $M^\prime$ and hence $M$ is Seifert-fibered.

We summarize the discussion as follows.

\begin{prop} Let $$ 1 \longrightarrow N \stackrel{i}{\longrightarrow} G
\stackrel{q}{\longrightarrow} Q \longrightarrow 1$$  be an exact
sequence of
groups where $Q = \pi_1(M)$ is an infinite non-virtually cyclic
 fundamental group of a closed  3-manifold $M$, $G$ is the fundamental group of an elliptic complex surface
 and $N$ is finitely generated. Then
\begin{enumerate}
\item $N$ must be virtually infinite cyclic, and
\item $M$ must be
 a Seifert-fibered space whose base orbifold is  flat or hyperbolic and a finite sheeted cover of $X$ must be
 a circle bundle over a finite sheeted cover of $M$.
\end{enumerate}
   \label{ellprop} \end{prop}

Combining Proposition \ref{ellprop} with Theorem \ref{main7} we have the
second main theorem of our paper.

\begin{theorem}
Let $$ 1 \longrightarrow N \stackrel{i}{\longrightarrow} G
\stackrel{q}{\longrightarrow} Q \longrightarrow 1$$  be an exact
sequence of
groups such that
\begin{enumerate}
\item[(a)] $Q = \pi_1(M)$ is
the fundamental group of a  closed 3-manifold $M$ such that $Q$ is
infinite and not virtually cyclic,
\item[(b)]  $G$ is the fundamental group of a compact complex surface,
and
\item[(c)] $N$ is finitely presented.
\end{enumerate}
Then there exists a finite index subgroup $Q^\prime$ of $Q$ such that
$Q^\prime$ is the fundamental group of a Seifert-fibered 3-manifold with
hyperbolic or flat base. Also, there exists
a finite index subgroup $G^\prime$ of $G$ such that $G^\prime$ is the fundamental group of an elliptic complex surface $X$
which is
 a circle bundle over a Seifert-fibered 3-manifold.
\label{main2} \end{theorem}

Combining Theorem \ref{main2} with Proposition \ref{ellprop} and using the fact that Seifert fibered spaces
with hyperbolic or flat base have finite sheeted covers with $b_1 > 1$,
we obtain the following theorem of Kotschick.

\begin{theorem}[\cite{kotschick}]\label{dsk2}
Let $Q$ be the fundamental group of a closed 3-manifold. Suppose further
that $Q$ is infinite and not virtually cyclic.
Then $Q$ cannot  be the fundamental group of a compact complex surface.
\end{theorem}

\section{Further Consequences}

\subsection{Products of 3 Manifolds and Circles}
Let $$ 1 \longrightarrow N \stackrel{i}{\longrightarrow} G
\stackrel{q}{\longrightarrow} Q \longrightarrow 1$$  be an exact
sequence of
groups where $N$ is finitely presented, $Q = \pi_1(M)$ is
the fundamental group of a closed Seifert-fibered  3-manifold $M$ whose base orbifold is  flat or hyperbolic,
and $G$ is the fundamental group of a compact  complex surface.  By Theorem \ref{main2},
$X$ must be an aspherical elliptic complex surface and
 $N$ must be virtually infinite cyclic. We now briefly discuss complex structures on circle bundles over
$M$, where $M$ is Seifert-fibered.

\begin{theorem} Let M be a closed orientable 3 manifold. Then $M \times S^1$ admits a complex structure if
and only if $M$ is Seifert fibered.
\label{3times} \end{theorem}

\begin{proof}
We will first show that every circle bundle over a Seifert-fibered
3-manifold admits a complex structure
(cf. \cite{wood}). We include a proof of it for completeness.

Let $X$ be a compact connected Riemann surface, and
let $T\,=\, {\mathbb C}/({\mathbb Z}\oplus \sqrt{-1}\cdot{\mathbb Z})$
be the elliptic curve. Consider
the short exact sequence of sheaves on $X$
\begin{equation}\label{e1}
0 \longrightarrow\, {\mathbb Z}\oplus \sqrt{-1}\cdot{\mathbb Z}
\,\longrightarrow
\, {\mathcal O}_X \, \longrightarrow\, \underline{T} \,
\longrightarrow\, 0\, ,
\end{equation}
where $\underline{T}$ is the sheaf of holomorphic functions to $T$, and
${\mathcal O}_X$ is the sheaf of holomorphic functions. Since
$H^2(X,\, {\mathcal O}_X)\,=\, 0$ (as $\dim_{\mathbb C}X\,=\, 1$), the
homomorphism
$$
H^1(X,\, \underline{T})\, \longrightarrow\, H^2(X,\, {\mathbb Z}\oplus
\sqrt{-1}\cdot{\mathbb Z})
$$
in the long exact sequence of cohomologies associated to the short exact sequence
in \eqref{e1} is
surjective. This implies that any $S^1\times S^1$-bundle over $X$
admits a complex structure.

Let $M\, \longrightarrow\, X$ be a Seifert-fibered 3-manifold
with (base) orbifold points $\{p_i\}_{i=1}^n$; let $m_i$ be the order of $p_i$.
Then a circle bundle over $M$ is diffeomorphic to a manifold obtained
by performing logarithmic transformations on a holomorphic principal
$T$-bundle over $X$; the logarithmic transformations are done over
the points $p_i$ using the automorphism of $T$ given by
the automorphism of $\mathbb C$ defined by $z\, \longmapsto\, z+1/m_i$.
(See \cite[Ch. V, \S~13]{bhpv} for logarithmic transformation.)
Therefore, we conclude that every circle bundle over a Seifert-fibered
surface admits a complex structure.

The converse direction follows from Theorems \ref{main1}, \ref{main2}
and Proposition \ref{ellprop}.\end{proof}

\subsection{Quasi-isometries}
\begin{theorem}[\cite{gersten-ijac}]   Two of the eight
3-dimensional geometries, $\widetilde{Sl_2 (\mathbb{R})}$ and
${\mathbb{H}}^2 \times \mathbb{R}$, are quasi-isometric.\label{qi}
\end{theorem}

Hence the unit tangent bundle $U(S)$ of a closed surface $S$ of genus
greater than one and the product $S \times S^1$ have quasi-isometric
fundamental groups. It follows that $H_1= \pi_1 (U(S) \times S^1)$ and $H_2= \pi_1 (S \times S^1 \times S^1)$ are quasi-isometric.
 Further note that $H_1$ is not commensurable with a K\"ahler group. This can be seen as follows: If there is a finite extension $K$ of $H_1$ which is K\"ahler then clearly
$H_1$ is K\"ahler. On the other hand, neither $H_1$ nor any finite index
subgroup $H$ of $H_1$ is K\"ahler: If $X$ is the finite cover of $U(S) \times S^1$ corresponding to  $H$
then $X = M \times S^1$ where $M$ is also a nontrivial circle bundle over a surface $S'$. Now a presentation of $\pi_1(M)$ is given by
$$ \langle a_1,b_1,\cdots ,a_g,b_g,t \ \vert \ \Pi_{i=1}^g[a_i,b_i]=t^k
\rangle $$
where $\{a_1,b_1,\cdots ,a_g,b_g\}$ is the usual set of generators for
$\pi_1(S')$ and $k$ is a nonzero integer.

Abelianizing one sees that $b_1(X) = b_1(S')+1$ and hence $b_1(H)$ is odd. $H_2$ is clearly K\"ahler being the fundamental group of the product of a torus and an orientable 2-manifold.
Thus we have an example demonstrating the fact that K\"ahlerness is not preserved, even up to commensurability, by quasi-isometries within the category of
compact aspherical complex surfaces.  This gives a strong negative answer to a
question of Gromov's (Problem on page 209 of \cite{gromov-ai}).

\subsection{Generalizations and Questions} After the completion of this paper, Thomas Delzant told us of the following simple 
and elegant proof  of a generalization
of Theorem \ref{dsk}:

\begin{theorem} \label{thdel} (Delzant) Let $G$ be a Poincar\'e duality group of dimension three
(PD(3) for short) group that does not have property $T$. Then $G$ cannot be K\"ahler. \end{theorem}

\begin{proof} We argue by contradiction.
Suppose $G$ is  a K\"ahler group and $M$ is a compact K\"ahler manifold with fundamental group $G$. 
One can find a unitary representation $\pi$ with first
reduced $l_2$ cohomology $H^1_{red}(G,\pi) \neq 0$. 
Let $\omega$ be the  harmonic representative of such a non-vanishing cohomology class.
Then  $\omega \wedge \overline{ \omega}$
 is a closed (1,1) form on the K\"ahler manifold $M$ which comes from $H^2(G)$. Since $G$ is PD(3), it follows that $H^1(G)$ is not trivial. 
Hence the Albanese map from $M$ to the Albanese torus is not trivial.
The image of this map must be a complex curve; else one would have a non-vanishing cohomology class in $H^4(G)$.

But this curve represent a class in $H^2( G)$ such that every class in $H^1(G)$ has trivial cup product with it, contradicting 
Poincar\'e duality. \end{proof}

\smallskip

One of the conclusions of Theorem \ref{main1} is that the quotient 3-manifold group
$Q$ could be the $3$-dimensional Heisenberg group. It is not clear
whether this situation can arise
at all. Thus we ask

\begin{qn} Do there exist  exact sequences $$1 \longrightarrow N
\longrightarrow G \longrightarrow Q \longrightarrow 1$$
 with $Q$ the $3$-dimensional Heisenberg group, $G$ a K\"ahler group and
$N$ finitely presented (or finitely generated)? \label{qq} \end{qn}

Question \ref{qq} seems to be related to the issue of finding which nilpotent groups may arise
as K\"ahler groups. Campana \cite{campana} gives examples of nilpotent
K\"ahler  groups $H$ fitting into the exact sequence $$1 \longrightarrow
\ints \longrightarrow H \longrightarrow {\ints}^{\oplus 8}
\longrightarrow 1$$
with the presentation $\{ t, x_1, y_1, \cdots , x_4, y_4 : [x_i,y_j] = \delta_{ij} t; [x_i, x_j]
= 0 = [y_i, y_j] \}$. However, it is not clear which 3-step nilpotent groups may arise
as K\"ahler groups. This seems to be the principal difficulty in addressing
Question \ref{qq} at the moment.

We end with a simple observation based on Campana's example above. It follows
from the presentation of the group that the $3$-dimensional Heisenberg
group can in fact
arise as a normal subgroup of a K\"ahler group. Let $H_1 \subset H$ be the subgroup
generated by $\{ t, x_1, y_1\}$. Then $H_1$ is isomorphic to the
$3$-dimensional Heisenberg group. Also the quotient group $H/H_1 =
{\ints}^{\oplus 6}$. Thus
we may have non-K\"ahler finitely presented
groups $N$ (here the $3$-dimensional Heisenberg group) appearing in
an exact sequence $$1 \longrightarrow N \longrightarrow G
\longrightarrow Q \longrightarrow 1$$ with both $G$ and $Q$
K\"ahler.

\bigskip

{\bf Acknowledgments:} This work was started during a visit of the third author to RKM Vivekananda University and completed during a visit of
the second author to Tata Institute of Fundamental Research, Mumbai. We thank these institutions for their hospitality. We thank Jonathan Hillman
for bringing the paper \cite{hill} to our notice when an earlier version of this paper was posted on arxiv; and Thomas Delzant for sharing with
us his simple and elegant (unpublished) proof of
Theorem \ref{thdel}. We would also like to thank
the referee for suggesting a generalization of the main result of an earlier version, for a substantial simplification
of the proof of Proposition \ref{remark-torus} and several expository changes.


\begin{thebibliography}{10}

\bibitem[ABCKT96]{abckt}
J.~Amoros, M.~Burger, K.~Corlette, D.~Kotschick, and D.~Toledo.
\newblock {Fundamental groups of compact K\"ahler manifolds}.
\newblock Mathematical Surveys and Monographs 44, American Mathematical
  Society, Providence, RI, 1996.

\bibitem[Ara09]{arapura-hom}
D.~Arapura.
\newblock {Homomorphisms between K\"ahler groups}.
\newblock {\em preprint, arXiv:0908.0929}, 2009.

\bibitem[BE73]{be}
R.~Bieri and B.~Eckmann.
\newblock {Groups with homological duality generalizing Poincar\'e
duality}. {\em Invent. Math.} \textbf{20} (1973), 103--124.

\bibitem[BHPV04]{bhpv}
W.~Barth, K.~Hulek, C.~Peters, and A.~van~de Ven.
\newblock Compact complex surfaces (2nd enlarged edition).
\newblock Ergebnisse der Mathematik und ihrer Grenzgebiete, 3. Folge,
  Volume 4. Springer Verlag, Berlin, 2004.

\bibitem[Bon02]{bon-note}
F.~Bonahon.
\newblock Geometric structures on 3-manifolds,
\newblock in: Handbook of Geometric Topology (R. Daverman, R. Sher
  eds.), pp. 93--164, Elsevier, 2002.

\bibitem[Bow97]{bowditch-relhyp}
B.~H. Bowditch.
\newblock Relatively hyperbolic groups.
\newblock {\em preprint, Southampton}, 1997.

\bibitem[Bru03]{brudnyi-solvq}
A. Brudnyi.
\newblock {Solvable quotients of K\"ahler groups}.
{\em Michigan Math. J.} \textbf{51} (2003), 477--490.

\bibitem[Cam95]{campana}
F.~Campana.
\newblock {Remarques sur les groupes de K\"ahler nilpotents}.
{\em Ann. Scient. Ec. Norm. Sup.} \textbf{28} (1995), 307--316.

\bibitem[CJ94]{casson-j}
A.~Casson and D.~Jungreis.
\newblock {Convergence groups and Seifert fibered 3-manifolds}.
{\em Invent. Math.} \textbf{118} (1994), 441--456.

\bibitem[CT89]{carlson-toledo-pihes}
J.~A. Carlson and D.~Toledo.
\newblock {Harmonic mappings of K\"ahler manifolds to locally symmetric
spaces}.
{\em Publ. Math. I.H.E.S.} \textbf{69} (1989), 173--201.

\bibitem[CT97]{carlson-toledo-vii}
J.~A. Carlson and D.~Toledo.
\newblock {On fundamental groups of Class VII surfaces}.
{\em Bull. London Math. Soc.} \textbf{29} (1997),
98--102.

\bibitem[DG05]{dg}
T. Delzant and M. Gromov.
\newblock Cuts in K\"ahler groups,
\newblock in: Infinite groups: geometric,
  combinatorial and dynamical aspects, ed. L. Bartholdi, Progress in
Mathematics, Vol. 248,
  Birkhauser Verlag Basel/Switzerland, pp. 31--55, 2005.




\bibitem[De10]{delz}
T. Delzant.
\newblock L{'}invariant de Bieri-Neumann-Strebel des groupes
fondamentaux des vari\'et\'es k\"ahl\'eriennes.
{\em  Math. Annalen} \textbf{348} (2010),119-125.

\bibitem[DP10]{delzant-py}
T. Delzant and P. Py.
\newblock {Kahler groups, real hyperbolic spaces and the Cremona group}.
\newblock {\em preprint, arXiv:1012.1585}, 2010.

\bibitem[DS09]{ds}
A.~Dimca and A.~I. Suciu.
\newblock {Which 3-manifold groups are K\"ahler groups?}
{\em J. Eur. Math. Soc.} \textbf{11} (2009), 521--528.

\bibitem[Far98]{farb-relhyp}
B.~Farb.
\newblock Relatively hyperbolic groups.
\newblock {\em Geom. Funct. Anal.} \textbf{8} (1998), 810--840.






\bibitem[FM94]{fm}
R. Friedman and J. W. Morgan.
\newblock Smooth four-manifolds and complex surfaces.
\newblock Springer, 1994.

\bibitem[FV09]{fv}
S. Friedl and S. Vidussi.
\newblock Twisted Alexander polynomials and fibered 3-manifolds.
\newblock To appear in the Proceedings of the Georgia International Topology Conference 2009.

\bibitem[Gab92]{gabai-cgnce}
D.~Gabai.
\newblock {Convergence groups are Fuchsian groups}.
{\em Ann. of Math.} \textbf{136} (1992), 447--510.

\bibitem[Ger92]{gersten-ijac}
S.~M. Gersten.
\newblock Bounded cocycles and combings of groups.
{\em Int. J. of Algebra and Computation} \textbf{2} (1992), 307--326.

\bibitem[GMRS98]{gmrs}
R.~Gitik, M.~Mitra, E.~Rips, M.~Sageev.
\newblock Widths of subgroups. 
{\em Trans. Amer. Math. Soc.} \textbf{350} (1998), no. 1, 321 -329.

\bibitem[Gro85]{gromov-hypgps}
M.~Gromov.
\newblock Hyperbolic groups,
\newblock in: Essays in Group Theory, ed. Gersten, pp. 75-263, MSRI Publ.,vol.8, Springer Verlag, 1985.

\bibitem[Gro93]{gromov-ai}
M.~Gromov.
\newblock Asymptotic invariants of infinite groups,
\newblock in: Geometric Group Theory, vol.2, pp. 1--295, Lond. Math.
Soc. Lecture Notes 182, Cambridge University Press, 1993.

\bibitem[GrS92]{grs}
M. Gromov and R. Schoen.
\newblock Harmonic maps into singular spaces and p-adic superrigidity for lattices in groups of rank one.
{\em Publ. Math. IHES} \textbf{76} (1976),  165–246.


\bibitem[Hem76]{hempel-book}
J.~Hempel.
\newblock 3 manifolds.
\newblock {Annals of Math. Studies No. 86, Princeton University Press},
  1976.

\bibitem[Hem87]{hempel-rf}
J.~Hempel.
\newblock Residual finiteness for 3-manifolds,
\newblock {in: Combinatorial groups theory and topology, ed. S. M.
Gersten and J. R. Stallings, pp. 379--396, Annals of Math. Studies vol.
111, Princeton Univ. Press}, 1987.

\bibitem[Hi98]{hill}
J.~Hillman.
\newblock  {On 4-dimensional mapping tori and product geometries.}
{\em  J. London Math. Soc.} (2) \textbf{58} (1998), no. 1, 229-238.

\bibitem[HW09]{hr-w}
G. C.~Hruska and D. T.~Wise.
\newblock Packing subgroups in relatively hyperbolic groups.
{\em Geom. Topol.} \textbf{13} (2009), no. 4, 1945–1988.

\bibitem[JS79]{js}
W.H.~Jaco, and P.B.~Shalen.
\newblock {Seifert Fibered Spaces in 3-manifolds}.
\newblock Mem. Amer. Math.  Soc. No. 220, Providence, RI, 1979.


\bibitem[KL98]{kl-npc}
M.~Kapovich and B.~Leeb.
\newblock 3-manifold groups and nonpositive curvature.
{\em Geom. Funct. Anal.} \textbf{8} (1998), 841--852.

\bibitem[Kot10]{kotschick}
D.~Kotschick.
\newblock {Three-manifolds and K\"ahler groups}.
\newblock {\em preprint, arXiv:math 1011.4084}, 2010.

\bibitem[Lue88]{luecke-vbetti}
J. Luecke.
\newblock {Finite covers of 3-Manifolds containing essential tori}.
{\em Trans. Amer. Math. Soc.} \textbf{310} (1998), 521--524.



\bibitem[Mj06]{mahan-split}
M. Mj.
\newblock {Cannon-Thurston maps for surface groups}.
\newblock {\em preprint, arXiv:math.GT/0607509}, 2006.

\bibitem[Mj08]{mahan-agt}
M. Mj.
\newblock {Relative rigidity, quasiconvexity and {$C$}-complexes}.
{\em Algebr. Geom. Topol.} \textbf{8} (2008), 1691--1716.

\bibitem[MS09]{mahan-mbdl}
M. Mj and P. Sardar.
\newblock {A combination theorem for metric bundles}.
{\em arXiv:0912.2715, submitted to Geom. Funct. Anal.}, 2009.




\bibitem[Per02]{Per1} G. Perelman. The entropy formula for the Ricci
flow and its geometric applications, arXiv:math/0211159.

\bibitem[Per03a]{Per2} G. Perelman. Ricci flow with surgery on three-manifolds,
arXiv:math/0303109.

\bibitem[Per03b]{Per3} G. Perelman. Finite extinction time for the solutions
to the Ricci flow on certain three-manifolds, arXiv:math/0307245.

\bibitem[Sam86]{sampson2}
J.~H. Sampson.
\newblock Applications of harmonic maps to K\"ahler geometry,
\newblock {in: Complex differential geometry and nonlinear differential
  equations (Brunswick, Maine, 1984), Contemp. Math.~{\bf 49}, Amer. Math.
  Soc., Providence, RI}, pp. 125--134, 1986.

\bibitem[Sco83a]{scott-geoms}
P.~Scott.
\newblock The geometries of 3-manifolds.
{\em Bull. London Math. Soc.} \textbf{15} (1983), 401--487.

\bibitem[Sco83b]{scott-sf}
P.~Scott.
\newblock There are no fake {S}eifert fibered spaces with infinite
  $\pi_1$. {\em Ann. Math.} \textbf{117} (1983), 35--70.

\bibitem[Se97]{zs}
Z.~Sela.
\newblock Acylindrical accessibility for groups. 
{\em Invent. Math.} \textbf{129} (1997), no. 3, 527--565.


\bibitem[Wa86]{wall}
C. ~T.~C.~Wall.
\newblock Geometric structures on compact complex surfaces. {\em Topology} \textbf{25}  (1986), 119--153.


\bibitem[Wi04]{wise}
D.~Wise.
\newblock Cubulating small-cancellation groups.
{\em Geom. Funct. Anal.} \textbf{14} (2004), 150—214.

\bibitem[Wo90]{wood}
J.~C. Wood.
\newblock {Harmonic morphisms, conformal foliations and Seifert fibre
spaces},
\newblock {\em in: Geometry of low-dimensional manifolds, 1 (Durham,
1989),
  London Math. Soc. Lecture Note Ser. 150, Cambridge Univ. Press}, pp.
  41--89, 1990.

\end{thebibliography}
\end{document}